\documentclass[10pt, reqno]{amsart}

\def\BibTeX{{\rm B\kern-.05em{\sc i\kern-.025em b}\kern-.08em
    T\kern-.1667em\lower.7ex\hbox{E}\kern-.125emX}}

\newtheorem{thm}{Theorem}[section]
\newtheorem{lem}[thm]{Lemma}
\newtheorem{prop}[thm]{Proposition}
\newtheorem{cond}[thm]{Condition}
\theoremstyle{definition}

\theoremstyle{remark}
\newtheorem{rem}{Remark}[section]

\numberwithin{equation}{section}

    \newcommand{\floor}[1]{\lfloor#1\rfloor}

    \newcommand{\eqd}{\stackrel{d}{=}}

\newcommand{\be}{\begin{equation}}
    \newcommand{\ee}{\end{equation}}

\begin{document}

\title[On the $J_{1}$ convergence for partial sum processes]{On the $J_{1}$ convergence for partial sum processes with a reduced number of jumps}

%\runtitle{On the $J_{1}$ convergence for partial sum processes with a reduced number of jumps}

%%
% See the complete version of this file, testart.tex, for more
% complicated examples
%
\author{Danijel Krizmani\'{c}}

\address{Danijel Krizmani\'{c}\\ Department of Mathematics\\
        University of Rijeka\\
        Radmile Matej\v{c}i\'{c} 2, 51000 Rijeka\\
        Croatia}
%% Note the doubled @@:
\email{dkrizmanic@math.uniri.hr}


%\maketitle

\begin{abstract}
Various functional limit theorems for partial sum processes of strictly stationary sequences of regularly varying random variables in the space of c\`{a}dl\`{a}g functions $D[0,1]$ with one of the Skorohod topologies have already been obtained. The mostly used Skorohod $J_{1}$ topology is inappropriate when clustering of large values of the partial sum processes occurs. When all extremes within each cluster of high-threshold excesses do not have the same sign, Skorohod $M_{1}$ topology also becomes inappropriate. In this paper we alter the definition of the partial sum process in order to shrink all extremes within each cluster to a single one, which allow us to obtain the functional $J_{1}$ convergence. We also show that this result can be applied to some standard time series models, including the GARCH(1,1) process and its squares, the stochastic volatility models and $m$--dependent sequences.
\end{abstract}

\maketitle

\section{Introduction}
\label{intro}

Let $(X_n)_{n\geq 1}$ be a strictly stationary sequence of real valued random variables and define by $S_n=X_1+\cdots +X_n,\ {n\geq 1}$, its accompanying sequence of partial
sums. If the sequence $(X_{n})$ is i.i.d.~then it is well known (see for example Gnedenko and Kolmogorov~\cite{GnKo54}, Rva\v{c}eva~\cite{Rv62}, Feller~\cite{Fe71}) that there exist real sequences $(a_n)$ and $(b_n)$ such that
\begin{equation}\label{e:CLT}
 \frac{S_{n}-b_{n}}{a_{n}} \xrightarrow{d} S \qquad \textrm{as} \ n \to \infty,
\end{equation}
for some non-degenerate $\alpha$--stable random variable $S$ with $\alpha \in (0,2)$ if and only if $X_{1}$ is regularly varying with index $\alpha \in (0,2)$, that is,
\begin{equation}\label{e:regvarosn2}
\mathrm{P} (|X_{1}| > x) = x^{-\alpha} L(x),
\end{equation}
where $L(\,\cdot\,)$ is a slowly varying function at $\infty$ and
\begin{equation}\label{e:regvarosn1}
 \frac{\mathrm{P}(X_{1}>x)}{\mathrm{P}(|X_{1}|>x)} \to p \qquad \textrm{and} \qquad \frac{\mathrm{P}(X_{1}< -x)}{\mathrm{P}(|X_{1}|>x)} \to q,
\end{equation}
as $x \to \infty$, with $p \in [0,1]$ and $q=1-p$. As $\alpha$ is less than 2, the variance of $X_{1}$ is infinite.

The functional generalization of (\ref{e:CLT}) has been studied extensively in probability literature. Define the partial sum processes
\begin{equation*}
V_{n}(t) =  \frac{1}{a_{n}} \sum_{k=1}^{\floor{nt}}(X_{k} - b_{n}), \qquad t \in [0,1],
\end{equation*}
where the sequences $(a_{n})$ and $(b_{n})$ are chosen as
\begin{equation*}
 n \mathrm{P}( |X_{1}| > a_{n}) \to 1 \qquad \textrm{and} \qquad b_{n} = \mathrm{E} \bigl( X_{1} \, 1_{\{ |X_{1}| \le a_{n} \}} \bigr).
\end{equation*}
Here $\floor{x}$ represents the integer part of the real number $x$. In functional limit theory one investigates the asymptotic behavior of the processes $V_{n}(\,\cdot\,)$ as $n \to \infty$. Since the sample paths of $V_{n}(\,\cdot\,)$ are elements of the space $D[0,1]$ of all right-continuous real valued functions on $[0,1]$ with left limits, it is natural to consider the weak convergence of distributions of $V_{n}(\,\cdot\,)$ with the one of Skorohod topologies on $D[0,1]$ introduced in Skorohod \cite{Sk56}.

Skorohod \cite{Sk57} established a functional limit theorem for the processes $V_{n}(\,\cdot\,)$ for infinite variance i.i.d.~regularly varying sequences $(X_{n})$. Under some weak dependence conditions, weak convergence of partial sum processes were obtained by Leadbetter and Rootz\'{e}n~\cite{LeRo88} and Tyran-Kami\'{n}ska~\cite{TK10}. Their functional limit theorems hold in Skorohod $J_{1}$ topology, which is appropriate when large values of the partial sum processes do not cluster. When clustering of large values occurs then $J_{1}$ convergence fails to hold, but the functional limit theorem might still hold in the weaker Skorohod $M_{1}$ topology. Avram and
Taqqu \cite{AvTa92} obtained a functional limit theorem with Skorohod $M_1$ topology for sums of moving
averages with nonnegative coefficients. Recently Basrak et al.~\cite{BKS} gave sufficient conditions for functional limit theorem with $M_{1}$ topology to hold for stationary,
regularly varying sequences for which all extremes within each cluster of high-threshold
excesses have the same sign.

In this paper we alter the definition of the partial sum process in the manner that all extremes within each cluster shrink to a single one, which allows us to recover the $J_{1}$ convergence. Note that the process $V_{n}(t)$ jumps at every $t=k/n$ ($k \leq n$), with $(X_{k}-b_{n})/a_{n}$ being the size of the jump. Now we reduce the number of jumps (or alternatively increase the intervals between jumps) by introducing a sequence of positive integers $(r_{n})$ such that $r_{n} \to \infty$ and $k_{n}:= \floor{n/r_{n}} \to \infty$ as $n \to \infty$, and defining new partial sum processes
\begin{equation*}
 W_{n}(t) = \frac{1}{a_{n}} \sum_{k=1}^{\lfloor k_{n}t \rfloor}
(S_{r_{n}}^{k} - c_{n}), \qquad t \in [0,1],
\end{equation*}
where $S_{r_{n}}^{k} = X_{(k-1)r_{n}+1} + \ldots + X_{kr_{n}}$ ($k,\,n \in \mathbb{N}$), and $c_{n}$ are centering constants which will be specified later. The process $W_{n}(t)$ jumps at every $t=k/k_{n}$, with $(S_{r_{n}}^{k}-c_{n})/a_{n}$ being the size of the jump. In other words, we break $X_{1}, X_{2}, \ldots$ into blocks of $r_{n}$ consecutive random variables and treat the sums of random variables within each block as we treated single random variables $X_{i}$ in the process $V_{n}(\,\cdot\,)$. One jump of the process $W_{n}(\,\cdot\,)$ corresponds to $r_{n}$ consecutive jumps of the process $V_{n}(\,\cdot\,)$. In this way we have partially smoothed the trajectories of partial sum processes such that each cluster can consist of only one excess.

Functional limit theorems for the processes $V_{n}(\,\cdot\,)$ base on the regular variation property of $X_{1}$. Therefore in obtaining the functional limit theorem for the processes $W_{n}(\,\cdot\,)$ we need to impose a similar condition on $S_{r_{n}}$. For this purpose we will assume $S_{r_{n}}$ satisfies a certain large deviation condition (see relation (\ref{e:reg-var-new0}) in the sequel).

The paper is organized as follows. In Section~\ref{s:prel} we introduce some basic results on point processes and regular variation. We also describe precisely the large deviation condition that we impose on $S_{r_{n}}$. In Section~\ref{s:glav} we state and prove the functional limit theorem for the processes $W_{n}(\,\cdot\,)$ in the $J_{1}$ topology. Here we also discuss several examples of stationary sequences covered by our theorem. Finally, in Section~\ref{s:kraj} (Appendix) we prove that the mixing conditions used in our main theorem are implied by some conditions which are given in terms of the standardly used $\alpha$--mixing and $\rho$--mixing conditions.

\section{Preliminaries}\label{s:prel}
 At the beginning we introduce here some basic notions and results on point processes which will be used later on. For more background on the theory of point processes we refer to Kallenberg~\cite{Ka83}. Let $\mathbb{E} = \overline{\mathbb{R}} \setminus \{0\}$, where $\overline{\mathbb{R}}=[-\infty, \infty]$. For $x,y \in \mathbb{E}$ define
 \begin{equation}\label{e:metricrho}
    \rho (x,y) = \max \Big\{ \Big| \frac{1}{|x|} -\frac{1}{|y|} \Big|, |\textrm{sign}\,x - \textrm{sign}\,y| \Big\},
 \end{equation}
where $\textrm{sign}\,z = z/|z|$. With the metric $\rho$, $\mathbb{E}$ becomes a locally compact, complete and separable matric space. A set $B \subseteq \mathbb{E}$ is relatively compact if it is bounded away from origin, that is, if there exists $u>0$ such that $B \subseteq \mathbb{E} \setminus [-u,u]$. Denote by $\mathcal{B}(\mathbb{E})$ the $\sigma$--algebra generated by $\rho$--open sets. Let $M_{+}(\mathbb{E})$ be the class of all Radon measures on $\mathbb{E}$, i.e. all nonnegative measures that are finite on relatively compact subsets of $\mathbb{E}$. A useful topology for $M_{+}(\mathbb{E})$ is the vague topology which renders $M_{+}(\mathbb{E})$ a complete separable metric space. If $\mu_{n} \in M_{+}(\mathbb{E})$, $n \geq 0$, then $\mu_{n}$ converges vaguely to $\mu_{0}$ (written $\mu_{n} \xrightarrow{v} \mu_{0}$) if $\int f\,d \mu_{n} \to \int f\,d \mu_{0}$ for all $f \in C_{K}^{+}(\mathbb{E})$, where $C_{K}^{+}(\mathbb{E})$ denotes the class of all nonnegative continuous real functions on $\mathbb{E}$ with compact support. One metric that induces the vague topology is given by
\begin{equation}\label{e:vaguemetric}
 d_{v}(\mu_{1}, \mu_{2}) = \sum_{k=1}^{\infty} 2^{-k} \bigg( \bigg| \int_{\mathbb{E}} f_{k}(x)\,\mu_{1}(dx) - \int_{\mathbb{E}} f_{k}\,\mu_{2}(dx) \bigg| \wedge 1 \bigg), \qquad \mu_{1}, \mu_{2} \in M_{+}(\mathbb{E}),
\end{equation}
for some sequence of functions $f_{k} \in C_{K}^{+}(\mathbb{E})$, where $a \wedge b = \min \{a,b\}$. We call $d_{v}$ the vague metric.

A Radon point measure is an element of $M_{+}(\mathbb{E})$ of the form $m = \sum_{i}\delta_{x_{i}}$, where $\delta_{x}$ is the Dirac measure. Denote by $M_{p}(\mathbb{E})$ the class of all Radon point measures. Since $M_{p}(\mathbb{E})$ is a subset of $M_{+}(\mathbb{E})$, we endow it with the relative topology. Let $\mathcal{M}_{p}(\mathbb{E})$ be the Borel $\sigma$--field of subsets of $M_{p}(\mathbb{E})$ generated by open sets. A point process on $\mathbb{E}$ is a measurable map from a given probability space to the measurable space $(M_{p}(\mathbb{E}), \mathcal{M}_{p}(\mathbb{E}))$. A standard example of point process is the Poisson process. Suppose $\mu$ is a given Radon measure on $\mathbb{E}$. Then $N$ is a Poisson process with mean (intensity) measure $\mu$, or synonymously, a Poisson random measure ($\mathrm{PRM}(\mu)$), if for all $A \in \mathcal{B}(\mathbb{E})$:
\begin{equation*}
 \mathrm{P}(N(A)=k) = \left\{
                          \begin{array}{cc}
                            \exp (-\mu(A)) (\mu(A))^{k}/k! & \textrm{if} \ \mu(A) < \infty \\[0.2em]
                            0 & \textrm{if} \ \mu(A) = \infty \\
                          \end{array}
                        \right.
 \end{equation*}
 and if $A_{1}, \ldots, A_{k} \in \mathcal{B}(\mathbb{E})$ are mutually disjoint, then $N(A_{1}), \ldots, N(A_{k})$ are independent random variables.

 A sequence of point processes $(N_{n})$ on $\mathbb{E}$ converges in distribution to a point process $N$ on $\mathbb{E}$ (written $N_{n} \xrightarrow{d} N$) if $\mathrm{E} f(N_{n}) \to \mathrm{E} f(N)$ for every bounded continuous function $f \colon M_{p}(\mathbb{E}) \to \mathbb{R}$. The point processes convergence is characterized by convergence of Laplace functionals. Denote by $\mathcal{B}_{+}$ the set of bounded measurable functions $f \colon \mathbb{E} \to [0,\infty)$. For a point process $N$ on $\mathbb{E}$ the Laplace functional of $N$ is the nonnegative function on $\mathcal{B}_{+}$ given by
 \begin{equation*}
  \Psi_{N}(f) = \mathrm{E} e^{-N(f)}, \qquad f \in \mathcal{B}_{+},
 \end{equation*}
 where $N(f) =  \int_{\mathbb{E}} f(x) N(dx)$.
 Then it holds that given point processes $N_{n}$, $n \geq 0$,
 \begin{equation}\label{e:ppLf}
    N_{n} \xrightarrow{d} N_{0} \qquad \textrm{iff} \qquad \Psi_{N_{n}}(f) \to \Psi_{N_{0}}(f) \quad \textrm{for all} \ f \in C_{K}^{+}(\mathbb{E})
 \end{equation}
  (see Kallenberg~\cite{Ka83}, Theorem 4.2).

Let $(X_{n})$ be a strictly stationary sequence of regularly varying
random variables with index $\alpha \in (0,2)$, and
let $(a_{n})$ be
a sequence of positive real numbers such that $n
\mathrm{P}(|X_{1}|>a_{n}) \to 1$ as $n \to \infty$. Regular variation can be expressed in terms of vague convergence of measures on $\mathbb{E}$:
\begin{equation*}
 n  \mathrm{P}
    ( a_{n}^{-1} X_{1} \in \cdot\,) \xrightarrow{v} \mu(\,\cdot\,) \qquad \textrm{as} \ n \to \infty,
\end{equation*}
    the Radon measure $\mu$ on $\mathbb{E}$ being given by
\begin{equation*}
 \mu(dx) = \big( p\alpha x^{-\alpha -1}1_{(0,\infty)}(x) +
   q\alpha (-x)^{-\alpha -1}1_{(-\infty,0)}(x) \big)dx,
\end{equation*}
where $p$ and $q$ are as in (\ref{e:regvarosn1}).

Using standard regular variation arguments it can be shown that for every $\lambda >0$ it holds that
$ a_{\lfloor \lambda n \rfloor}/a_{n} \to \lambda^{1/\alpha}$ as $n \to \infty$.
Therefore $a_{n}$ can be represented as
$a_{n} = n^{1/\alpha} L'(n)$,
where $L'(\,\cdot\,)$ is a slowly varying function at $\infty$.

Through the whole paper we will assume the sequence $(X_{n})$ satisfies the following large deviation type relations:
\begin{equation}\label{e:large-dev}
 \begin{array}{rl}
    k_{n} \mathrm{P}(S_{r_{n}} > xa_{n}) & \to  \ c_{+}x^{-\alpha},\\[0.6em]
  k_{n} \mathrm{P}(S_{r_{n}} < -xa_{n})  & \to  \ c_{-}x^{-\alpha},
  \end{array} \qquad x>0,
\end{equation}
as $n \to \infty$, where $c_{+}, c_{-} \geq 0$ are some constants, $(r_{n})$ is a sequence of positive
integers such that $r_{n} \to \infty$ and $r_{n}/n \to 0$ as $n \to
\infty$, and $k_{n} = \lfloor n/r_{n} \rfloor$. Some sufficient
conditions for relations in (\ref{e:large-dev}) to hold are given in
Bartkiewicz et al. \cite{BaJaMiWi09} and Davis and Hsing
\cite{DaHs95}.
 It is easy to see (for example by Lemma 6.1 in Resnick
\cite{Re07})  that (\ref{e:large-dev}) is equivalent to
\begin{equation}\label{e:reg-var-new0}
    k_{n} \mathrm{P}
    ( a_{n}^{-1} S_{r_{n}} \in \cdot\,) \xrightarrow{v} \nu(\,\cdot\,) \qquad \textrm{as} \ n \to \infty,
\end{equation}
where $\nu$ is the measure
\begin{equation}\label{e:new-measure}
  \nu(dx) = \big( c_{+}\alpha x^{-\alpha -1}1_{(0,\infty)}(x) +
   c_{-}\alpha (-x)^{-\alpha -1}1_{(-\infty,0)}(x) \big)dx.
 \end{equation}

\smallskip

\begin{lem}\label{l:auxlemma1}
   Let $\alpha \in (0,1)$ and assume relation (\ref{e:reg-var-new0})
   holds. Then for any $u >0$,
   \begin{equation}\label{e:eqlimit11}
      \lim_{n \to \infty} k_{n} \mathrm{E} \bigg( \frac{|S_{r_{n}}|}{a_{n}}
      1 _{ \big\{ \frac{|S_{r_{n}}|}{a_{n}} \leq u \big\} } \bigg)
      = \int_{|x| \leq u} |x|\,\nu (dx).
   \end{equation}
\end{lem}

\begin{proof} Fix $u >0$.
 Define
 \begin{equation*}
 \nu_{n}(\,\cdot\,) = k_{n} \mathrm{P} ( a_{n}^{-1} S_{r_{n}} \in \cdot\,), \qquad n \in \mathbb{N},
 \end{equation*}
 and
 \begin{equation*}
  f_{\delta}(x) = |x| 1_{\overline{B}(\delta,\,u)}(x), \qquad x \in \mathbb{E},\,\delta \in (0, u),
 \end{equation*}
 where
 $B(\delta,\,u) = \{ x \in \mathbb{E} : \delta < |x| < u \}$
 (and $\overline{B}(\delta,\,u) = \{ x \in \mathbb{E} : \delta \leq |x| \leq u \}$).
 By relation (\ref{e:reg-var-new0}) we have
 $\nu_{n} \xrightarrow{v} \nu$ as $n \to \infty$, and this yields
 \begin{equation}\label{e:konv1}
    \int_{\mathbb{E}} f_{\delta}(x)\,\nu_{n}(dx) \rightarrow \int_{\mathbb{E}}
    f_{\delta}(x)\,\nu(dx),
 \end{equation}
 as $n \rightarrow \infty$ (see Kallenberg~\cite{Ka83}, 15.7.3).
 Define
 \begin{equation*}
  f(x) = |x| 1_{\overline{B}(u)}(x), \qquad x \in \mathbb{E},
 \end{equation*}
 where $B(r) = \{ x \in \mathbb{E} : |x| < r \}$.
 For any $\delta \in (0,\,u)$ it holds
 \begin{eqnarray}\label{e:konv2}
   \nonumber \bigg| \int_{\mathbb{E}} f(x)\,\nu_{n}(dx) - \int_{\mathbb{E}} f(x)\,\nu(dx) \bigg|  & \le & \bigg|
      \int_{B(\delta)} f(x)\,\nu_{n}(dx) - \int_{B(\delta)} f(x)\,\nu(dx) \bigg| \\[0.5em]
   \nonumber & \hspace*{-24em} &  \hspace*{-12em} \displaystyle + \ \bigg|
      \int_{B(\delta)^{c}} f(x)\,\nu_{n}(dx) - \int_{B(\delta)^{c}} f(x)\,\nu(dx) \bigg|\\[0.5em]
    \nonumber & \hspace*{-24em} \le &  \hspace*{-12em} \displaystyle \bigg|
      \int_{B(\delta)} f(x)\,\nu_{n}(dx) \bigg| + \bigg| \int_{B(\delta)} f(x)\,\nu(dx) \bigg|\\[0.5em]
    & \hspace*{-24em} &  \hspace*{-12em} \displaystyle + \ \bigg| \int_{\overline{B}(\delta,\,u)}
      f(x)\,\nu_{n}(dx) - \int_{\overline{B}(\delta,\,u)} f(x)\,\nu(dx) \bigg|.
 \end{eqnarray}
 For the first term on the right
hand side of (\ref{e:konv2}) we have
\begin{eqnarray}
   \nonumber  \bigg| \int_{B(\delta)} f(x)\,\nu_{n}(dx) \bigg| & = &
       \int_{\mathbb{E}} |x| 1_{B(\delta)}(x)\,\nu_{n}(dx) = k_{n} \int \bigg| \frac{S_{r_{n}}}{a_{n}} \bigg|
       1_{\{ |S_{r_{n}}| < \delta a_{n} \}}\,d \mathrm{P} \\[0.6em]
   \nonumber & \hspace*{-14em} = &  \hspace*{-7em} \displaystyle k_{n} \mathrm{E} \bigg[ \frac{|S_{r_{n}}|}{a_{n}}
       1_{\{ |S_{r_{n}}| < \delta a_{n} \}} \bigg] = k_{n} \mathrm{E} \bigg[ \frac{|S_{r_{n}}|}{a_{n}}
       1_{\{ |S_{r_{n}}| < \delta a_{n} \}} 1_{\{ \cap_{j=1}^{r_{n}} \{ |X_{j}| \leq
       \delta a_{n} \} \}} \bigg]  \\[0.6em]
   \nonumber & \hspace*{-14em} &  \hspace*{-7em} \displaystyle + \
       k_{n} \mathrm{E}
       \bigg[ \frac{|S_{r_{n}}|}{a_{n}}
       1_{\{ |S_{r_{n}}| < \delta a_{n} \}} 1_{\{ \cup_{j=1}^{r_{n}}
       \{ |X_{j}| > \delta a_{n} \} \}} \bigg].
  \end{eqnarray}
  This term is bounded above by
  \begin{eqnarray}\label{e:konv3}
    \nonumber & \leq &   \displaystyle k_{n} \mathrm{E}
       \bigg[ \frac{\sum_{j=1}^{r_{n}}|X_{j}|}{a_{n}}
       1_{\{ \cap_{j=1}^{r_{n}} \{ |X_{j}| \leq \delta a_{n} \} \}}
       \bigg] + k_{n} \delta \mathrm{P} \bigg( \bigcup_{j=1}^{r_{n}} \{ |X_{j}| > \delta a_{n} \}
       \bigg)\\[0.3em]
   \nonumber &  \leq  &  \displaystyle k_{n} \sum_{j=1}^{r_{n}} \mathrm{E}
       \bigg[ \frac{|X_{j}|}{a_{n}} 1_{\{ |X_{j}| \leq \delta a_{n} \}}
       \bigg] + k_{n} \delta \sum_{j=1}^{r_{n}} \mathrm{P} (|X_{j}| > \delta a_{n} )\\[0.3em]
   \nonumber &  =  &\displaystyle k_{n} r_{n} \mathrm{E}
       \bigg[ \frac{|X_{1}|}{a_{n}} 1_{\{ |X_{1}| \leq \delta a_{n} \}}
       \bigg] + k_{n} r_{n} \delta  \mathrm{P} ( |X_{1}| > \delta a_{n} )\\[0.3em]
   &  = & \displaystyle \delta \cdot \frac{k_{n} r_{n}}{n}
       \cdot n \mathrm{P} (|X_{1}| > \delta a_{n}) \cdot \bigg[ \frac{\mathrm{E}
      [|X_{1}| 1_{\{ |X_{1}| \leq \delta a_{n} \}}]}{\delta a_{n} \mathrm{P}(|X_{1}| > \delta a_{n})} +  1 \bigg].
\end{eqnarray}
 From the definition of the sequences $(r_{n})$ and $(k_{n})$ it follows
$ k_{n}r_{n}/n \to 1$ as $n \to \infty$.
Since $X_{1}$ is a regularly varying random variable with index
$\alpha$, it follows immediately
$ n \mathrm{P}(|X_{1}| > \delta a_{n} ) \to \delta^{- \alpha}$ as
   $n \to \infty$.
By Karamata's theorem it holds that
\begin{equation*}
 \lim_{n \rightarrow \infty} \frac{\mathrm{E}
    [|X_{1}| 1_{\{ |X_{1}| \leq \delta a_{n} \}}]}{\delta a_{n} \mathrm{P}(|X_{1}| > \delta a_{n})}
    = \frac{\alpha}{1 - \alpha}.
\end{equation*}
Now from (\ref{e:konv3}) we get
\begin{equation*}
  \limsup_{n \to \infty} \bigg| \int_{B(\delta)} f(x)\,\nu_{n}(dx) \bigg|
  \le \delta^{1- \alpha} \bigg( \frac{\alpha}{1- \alpha} +1 \bigg),
\end{equation*}
and therefore, since $\alpha \in (0,1)$,
\begin{equation}\label{e:konv4}
    \lim_{\delta \to 0} \limsup_{n \to \infty} \bigg|
    \int_{B(\delta)} f(x)\,\nu_{n}(dx) \bigg|  = 0.
\end{equation}
By the representation of the measure $\nu$ in (\ref{e:new-measure})
we get
\begin{equation*}
 \int_{|x| < \delta} |x|\,\nu(dx) = (c_{-} + c_{+}) \frac{\alpha}{1-\alpha}\,\delta^{1-\alpha}.
\end{equation*}
 Hence for the second term on the right hand side of (\ref{e:konv2}) we have
\begin{equation}\label{e:konv5}
     \bigg| \int_{B(\delta)} f(x)\,\nu(dx) \bigg| =
     \int_{|x|< \delta } |x| \,\nu(dx) \rightarrow 0 \qquad \textrm{as} \ \delta \rightarrow 0.
\end{equation}
From (\ref{e:konv1}) we get for the third term on the right hand
side of (\ref{e:konv2})
\begin{equation}\label{e:konv6}
   \bigg| \int_{\overline{B}(\delta,\,u)} f(x)\,\nu_{n}(dx) -
      \int_{\overline{B}(\delta,\,u)} f(x)\,\nu(dx) \bigg|
       =  \bigg|
      \int_{\mathbb{E}} f_{\delta}(x)\,\nu_{n}(dx) - \int_{\mathbb{E}} f_{\delta}(x)\,\nu(dx) \bigg| \rightarrow 0
\end{equation}
as $n \to \infty$.
Now from (\ref{e:konv2}) using (\ref{e:konv4}), (\ref{e:konv5}) and
(\ref{e:konv6}) we obtain
\begin{equation*}
 \lim _{\delta \to 0}\limsup_{n \to \infty} \bigg| \int_{\mathbb{E}} f(x)\,\nu_{n}(dx) - \int_{\mathbb{E}}
   f(x)\,\nu(dx) \bigg|=0.
\end{equation*}
From this immediately follows
\begin{equation*}
 \int_{\mathbb{E}} f(x)\,\nu_{n}(dx) \to
   \int_{\mathbb{E}} f(x)\,\nu(dx) \qquad \textrm{as} \ n \rightarrow \infty,
\end{equation*}
i.e.
\begin{equation*}
  k_{n} \mathrm{E} \bigg( \frac{|S_{r_{n}}|}{a_{n}} 1_{ \big\{ \frac{|S_{r_{n}}|}{a_{n}} \leq
  u
  \big\} } \bigg) \to \int_{|x| \leq u}|x|\,\nu(dx)
  \qquad \textrm{as} \ n \to \infty.
\end{equation*}
 \end{proof}

\section{Main theorem}\label{s:glav}

Let $(X_n)$ be a strictly stationary sequence of regularly varying random variables with index $\alpha \in (0, 2)$. Assume (\ref{e:large-dev}) holds. The theorem below gives conditions
under which a stochastic sum process constructed from the sequence $(S_{r_{n}}^{k})$ satisfies a nonstandard
functional limit theorem in the space $D[0,1]$ of real-valued c\`{a}dl\`{a}g functions equipped with the Skorohod $J_{1}$ topology, with a non-Gaussian $\alpha$--stable
L\'{e}vy process as a limit. Recall that the distribution of a
L\'{e}vy process $W(\,\cdot\,)$ is characterized by its
characteristic triplet, i.e.\ the characteristic triplet of the
infinitely divisible distribution of $W(1)$. The characteristic
function of $W(1)$ and the characteristic triplet $(a, \mu, b)$ are
related in the following way:
\begin{equation*}
 \mathrm{E} [e^{izW(1)}] = \exp \biggl( -\frac{1}{2}az^{2} + ibz + \int_{\mathbb{R}} \bigl( e^{izx}-1-izx 1_{[-1,1]}(x) \bigr)\,\mu(dx) \biggr)
\end{equation*}
for $z \in \mathbb{R}$; here $a \ge 0$, $b \in \mathbb{R}$ are constants, and $\mu$ is a measure on $\mathbb{R}$ satisfying
\begin{equation*}
 \mu ( \{0\})=0 \qquad \text{and} \qquad \int_{\mathbb{R}}(|x|^{2} \wedge 1)\,\mu(dx) < \infty,
\end{equation*}
that is, $\mu$ is a L\'{e}vy measure. For a textbook treatment of
L\'{e}vy processes we refer to Bertoin~\cite{Bertoin96} and
Sato~\cite{Sa99}.

The metric $d_{J_{1}}$ that generates the $J_{1}$ topology on $D[0,1]$ is defined in the following way. Let $\Delta$ be the set of strictly increasing continuous functions
$\lambda \colon [0,1] \to [0,1]$ such that $\lambda(0)=0$ and
$\lambda(1)=1$, and let $e \in \Delta$ be the identity map on
$[0,1]$, i.e. $e(t)=t$ for all $t \in [0,1]$. For $x,y \in D[0,1]$
define
\begin{equation*}
    d_{J_{1}}(x,y) = \inf \{ \|x \circ \lambda - y \|_{[0,1]} \vee
    \| \lambda - e \|_{[0,1]} : \lambda \in \Delta \},
\end{equation*}
where $\|x\|_{[0,1]}=\sup \{ |x(t)| : t \in [0,1] \}$ and $a \vee b = \max \{a,b\}$.
Then $d_{J_{1}}$ is a metric on $D[0,1]$ and is called the Skorohod $J_{1}$
metric.

 The mixing condition
 appropriate for the result in this section is similar to the condition $\mathcal{A}(a_{n})$ of Davis and Hsing~\cite{DaHs95}, and hence we denote it by $\mathcal{A}^{*}(a_{n})$ and say that a strictly stationary sequence of random
variables $(X_{n})$
 satisfies the mixing condition $\mathcal{A}^{*}(a_{n})$ if there exist a sequence of positive
integers $(r_{n})$ such that $r_{n} \to \infty $ and $r_{n} / n \to
0$ as $n \to \infty$, and such that for every $f \in
C_{K}^{+}(\mathbb{E})$ (denoting $k_{n} = \lfloor n / r_{n}
\rfloor$)
\begin{equation}\label{e:mixcon-new}
   \mathrm{E} \exp \bigg( - \sum_{k=1}^{k_{n}}
              f(a_{n}^{-1}S_{r_{n}}^{k}) \bigg) - \bigg( \mathrm{E} \exp (-f(a_{n}^{-1}S_{r_{n}}))
              \bigg)^{k_{n}}   \to 0 \qquad \textrm{as} \ n \to \infty.
 \end{equation}
In case $\alpha \in [1,2)$, we will need to assume that the contribution of the smaller increments of the partial sum process is close to its expectation.

\begin{cond}\label{c:step6cond-new}
There exists a sequence of positive integers $(r_{n})$ with $r_{n}
\to \infty$ and $k_{n} = \lfloor n/r_{n} \rfloor \to \infty$ as $n
\to \infty$, such that for all $\delta
>0$,
\begin{equation*}
 \lim_{u \downarrow 0} \limsup_{n \to \infty}
              \mathrm{P} \bigg[ \max_{1 \leq j \leq k_{n}} \bigg| \sum_{k=1}^{j}
              \bigg( \frac{S_{r_{n}}^{k}}{a_{n}} 1_{ \big\{ \frac{|S_{r_{n}}^{k}|}{a_{n}}
              \leq u  \big\}} - \mathrm{E} \bigg( \frac{S_{r_{n}}^{k}}{a_{n}}
              1_{ \big\{ \frac{|S_{r_{n}}^{k}|}{a_{n}} \leq u \big\} } \bigg) \bigg) \bigg| > \delta
              \bigg]=0.
\end{equation*}
\end{cond}

%\smallskip

In Appendix we discuss some sufficient conditions for
the mixing condition $\mathcal{A}^{*}(a_{n})$ and Condition
\ref{c:step6cond-new} to hold.

\smallskip

\begin{thm}\label{t:newprocess}
  Let $(X_{n})$ be a strictly stationary sequence of regularly varying random variables
  with index $\alpha \in (0,2)$, and let $(a_{n})$ be a sequence of
positive real numbers such that $n \mathrm{P}(|X_{1}|>a_{n}) \to 1$
as $n \to \infty$. Suppose there exists a sequence of positive
integers $(r_{n})$ such that, as $n \to \infty$, $r_{n} \to \infty$,
$k_{n} = \lfloor n/r_{n} \rfloor \to \infty$ and
\begin{equation}\label{e:reg-var-new}
     k_{n} \mathrm{P}
\bigg( \frac{S_{r_{n}}}{a_{n}} \in \cdot
             \bigg) \xrightarrow{v} \nu(\,\cdot\,).
\end{equation}
Suppose that the mixing condition $\mathcal{A}^{*}(a_{n})$ and Condition \ref{c:step6cond-new} if $\alpha \in [1,2)$ hold with the same sequence $(r_{n})$ as in (\ref{e:reg-var-new}).
Then for a stochastic process defined by
\begin{equation*}
 W_{n}(t) = \sum_{k=1}^{\lfloor k_{n}t \rfloor}
\frac{S_{r_{n}}^{k}}{a_{n}} - \lfloor k_{n}t \rfloor \mathrm{E}
\bigg( \frac{S_{r_{n}}}{a_{n}} 1_{ \big\{ \frac{|S_{r_{n}}|}{a_{n}}
\le 1 \big\} }
    \bigg), \quad t \in [0,1],
\end{equation*}
it holds that
\begin{equation*}
 W_{n} \xrightarrow{d} W_{0}, \qquad n \to \infty,
\end{equation*}
in $D[0,1]$ endowed with the $J_{1}$ topology, where
$W_{0}(\,\cdot\,)$ is an $\alpha$--stable L\'{e}vy process with
characteristic triplet $(0,\nu,0)$.
\end{thm}

\begin{proof}
 Let, for any $n \in \mathbb{N}$,
 $(Z_{n,k})_{k}$ be a sequence of i.i.d.~random
 variables such that
 $Z_{n,1} \eqd S_{r_{n}}$.
 By relation (\ref{e:reg-var-new}) we have
\begin{equation}\label{e:distribS3}
   k_{n} \mathrm{P} \bigg(  \frac{Z_{n,1}}{a_{n}} \in \cdot
   \bigg) \xrightarrow{v} \nu (\,\cdot\,) \qquad
   \textrm{as} \ n \to \infty.
\end{equation}
Theorem 5.3 in Resnick~\cite{Re07} then implies, as $n \to \infty$,
\begin{equation}\label{e:lapfunct-new}
   \widetilde{\xi}_{n} := \sum_{k=1}^{k_{n}}
   \delta_{a_{n}^{-1}Z_{n,k}} \xrightarrow{d}
   \mathrm{PRM} (\nu)
   \end{equation}
 on $\mathbb{E}$. Define the point process
 $  \xi_{n} = \sum_{k=1}^{k_{n}} \delta_{a_{n}^{-1}S_{r_{n}}^{k}}$.
 For any $f \in C_{K}^{+}(\mathbb{E})$ we have
\begin{eqnarray*}
   \Psi_{\xi_{n}}(f) - \Psi_{\widetilde{\xi}_{n}}(f) & = &  \mathrm{E} \exp \bigg( - \sum_{k=1}^{k_{n}}
   f(a_{n}^{-1}S_{r_{n}}^{k}) \bigg) - \bigg( \mathrm{E} \exp (-f(a_{n}^{-1}Z_{1,\,n}))
   \bigg)^{k_{n}}\\[0.3em]
   & = &  \mathrm{E} \exp \bigg( - \sum_{k=1}^{k_{n}}
   f(a_{n}^{-1}S_{r_{n}}^{k}) \bigg) - \bigg( \mathrm{E} \exp (-f(a_{n}^{-1}S_{r_{n}}))
   \bigg)^{k_{n}}.
\end{eqnarray*}
Hence, the mixing condition $\mathcal{A}^{*}(a_{n})$ implies
$\Psi_{\xi_{n}}(f) - \Psi_{\widetilde{\xi}_{n}}(f) \to 0$ as $n \to
\infty$. Then by relations (\ref{e:ppLf}) and (\ref{e:lapfunct-new})  we obtain, as $n \to \infty$,
\begin{equation}\label{e:dodatak1}
 \sum_{k=1}^{k_{n}} \delta_{a_{n}^{-1}S_{r_{n}}^{k}} \xrightarrow{d}
    \mathrm{PRM}(\nu).
\end{equation}

 Suppose $U_{1}, \ldots, U_{k_{n}}$ are i.i.d. random variables uniformly distributed on $[0,1]$ with order statistics $U_{1:k_{n}} \leq U_{2:k_{n}} \leq \ldots \leq U_{k_{n}:k_{n}}$, which are independent of $(S_{r_{n}}^{k})$. From (\ref{e:dodatak1}) using Lemma 4.3 of Resnick~\cite{Re86} we obtain, as $n \to \infty$,
\begin{equation*}
\sum_{k=1}^{k_{n}} \delta_{(U_{k},\,a_{n}^{-1}S_{r_{n}}^{k})} \xrightarrow{d} \mathrm{PRM} (\mathbb{LEB} \times \nu).
\end{equation*}
From the independence of $(U_{k})$ and $(S_{r_{n}}^{k})$, we have that
\begin{equation*}
\sum_{k=1}^{k_{n}} \delta_{(U_{k:k_{n}},\,a_{n}^{-1}S_{r_{n}}^{k})} \eqd \sum_{k=1}^{k_{n}} \delta_{(U_{k},\,a_{n}^{-1}S_{r_{n}}^{k})}
\end{equation*}
as random elements of $M_{+}([0,1] \times \mathbb{E})$. Therefore
\begin{equation}\label{e:dodatak2}
\sum_{k=1}^{k_{n}} \delta_{(U_{k:k_{n}},\,a_{n}^{-1}S_{r_{n}}^{k})} \xrightarrow{d} \mathrm{PRM} (\mathbb{LEB} \times \nu).
\end{equation}
Using the arguments from Step 3 in the proof of Theorem 6.3 in Resnick~\cite{Re07} we get
\begin{equation}\label{e:dodatak3}
d_{v} \Big(  \sum_{k=1}^{k_{n}}
  \delta_{(k/ k_{n},\,a_{n}^{-1}S_{r_{n}}^{k})},\,\sum_{k=1}^{k_{n}} \delta_{(U_{k:k_{n}},\,a_{n}^{-1}S_{r_{n}}^{k})} \Big) \xrightarrow{P} 0 \qquad \textrm{as} \ n \to \infty,
\end{equation}
where $d_{v}$ is the vague metric on $M_{+}([0,1] \times \mathbb{E})$ (cf.~(\ref{e:vaguemetric})).
From (\ref{e:dodatak2}) and (\ref{e:dodatak3}) using Slutsky's theorem (see for instance Theorem 3.4 in Resnick~\cite{Re07}) we obtain, as $n \to \infty$,
\begin{equation*}\label{e:uvjetmix2-new}
  \sum_{k=1}^{k_{n}}
  \delta_{(k/ k_{n},\,a_{n}^{-1}S_{r_{n}}^{k})}
  \xrightarrow{d} \mathrm{PRM} (\mathbb{LEB} \times \nu) = \sum_{k} \delta_{(t_{k},j_{k})}
\end{equation*}
on $[0,1] \times \mathbb{E}$.
From this using the same arguments as in the proof of Theorem 7.1 in Resnick~\cite{Re07}, we obtain that, as $n \to \infty$,
\begin{equation}\label{e:slut11}
     W_{n}^{(u)} \xrightarrow{d} W_{0}^{(u)}
\end{equation}
in $D[0,1]$ with the $J_{1}$ topology, where
\begin{equation*}
  W_{n}^{(u)}(\,\cdot\,) := \sum_{k=1}^{\lfloor k_{n} \cdot \rfloor} \frac{S_{r_{n}}^{k}}{a_{n}}
     1_{\big\{ \frac{|S_{r_{n}}^{k}|}{a_{n}} > u
     \big\}} -  \lfloor k_{n} \,\cdot \, \rfloor
    \mathrm{E} \bigg( \frac{S_{r_{n}}}{a_{n}}
    1_{ \big\{ u < \frac{|S_{r_{n}}|}{a_{n}} \leq 1 \bigr\} }
    \biggr),
\end{equation*}
and
\begin{equation*}
  W_{0}^{(u)}(\,\cdot\,) := \sum_{t_{k} \leq \cdot}j_{k} 1_{ \{|j_{k}| > u \}} - (\,\cdot\,) \int_{u < |x| \leq 1}x\,\nu(dx).
\end{equation*}
From the L\'{e}vy-It\^{o}
representation of a L\'{e}vy process (see Section 5.5.3 in Resnick
\cite{Re07} or Theorem 19.2 in
Sato \cite{Sa99}), there exists a L\'{e}vy process
$W_{0}(\,\cdot\,)$ with  characteristic triplet $(0,\nu,0)$ such
that
\begin{equation*}
 \sup_{t \in [0,1]} |W_{0}^{(u)}(t) - W_{0}(t)| \xrightarrow{} 0
\end{equation*}
almost surely as $u \downarrow 0$. Since uniform convergence implies Skorohod $J_{1}$ convergence, we
get
$ d_{J_{1}}(W_{0}^{(u)}, W_{0}) \to 0$
almost surely as $u \downarrow 0$. Therefore since almost sure
convergence implies convergence in distribution,
\begin{equation}\label{e:slut12}
    W_{0}^{(u)} \xrightarrow{d} W_{0} \qquad
    \textrm{as} \ u \to 0,
\end{equation}
in $D[0,1]$ with the $J_{1}$ topology.

If we show that
\begin{equation*}
 \lim_{u \downarrow 0} \limsup_{n \to \infty}
   \mathrm{P}[d_{J_{1}}(W_{n}^{(u)}, W_{n}) > \delta]=0
\end{equation*}
 for any $\delta>0$, then from (\ref{e:slut11}), (\ref{e:slut12}) and Theorem 3.5 in Resnick \cite{Re07} we
 will have, as $n \to \infty$,
\begin{equation*}
  W_{n} \xrightarrow{d} W_{0}
\end{equation*}
 in $D[0,1]$ with the $J_{1}$ topology. Since the $J_{1}$
 metric on $D[0,1]$ is bounded above by the uniform metric on
 $D[0,1]$, it suffices to show that
 \begin{equation*}\label{e:step6-new}
    \lim_{u \downarrow 0} \limsup_{n \to \infty} \mathrm{P} \biggl(
 \sup_{t \in [0,1]} |W_{n}^{(u)}(t) - W_{n}(t)| >
 \delta \biggr)=0.
 \end{equation*}
 We have
 \begin{equation}\label{e:step6-new1}
 \mathrm{P} \biggl(
 \sup_{t \in [0,1]} |W_{n}^{(u)}(t) - W_{n}(t)| >
 \delta \biggr) = \mathrm{P} \bigg[ \max_{1 \leq j \leq k_{n}} \bigg| \sum_{k=1}^{j}
              \bigg( \frac{S_{r_{n}}^{k}}{a_{n}} 1_{ \big\{ \frac{|S_{r_{n}}^{k}|}{a_{n}}
              \leq u  \big\}} - \mathrm{E} \bigg( \frac{S_{r_{n}}^{k}}{a_{n}}
              1_{ \big\{ \frac{|S_{r_{n}}^{k}|}{a_{n}} \leq u \big\} } \bigg) \bigg) \bigg| > \delta
              \bigg].
 \end{equation}
 For $\alpha \in [1,2)$ this relation is simply Condition
 \ref{c:step6cond-new}. Therefore it remains to show
 (\ref{e:step6-new1}) for the case when $\alpha \in (0,1)$. Hence
assume $\alpha \in (0,1)$. For arbitrary (and fixed)
 $\delta >0$ define
\begin{equation*}
 I(u,n) =
    \mathrm{P} \bigg[ \max_{1 \leq j \leq k_{n}} \bigg| \sum_{k=1}^{j}
    \bigg( \frac{S_{r_{n}}^{k}}{a_{n}} 1_{ \big\{ \frac{|S_{r_{n}}^{k}|}{a_{n}} \leq u  \big\}}
    - \mathrm{E} \bigg( \frac{S_{r_{n}}^{k}}{a_{n}}
    1_{ \big\{ \frac{|S_{r_{n}}^{k}|}{a_{n}} \leq u \big\} } \bigg) \bigg) \bigg| > \delta
    \bigg].
\end{equation*}
Using stationarity and Chebyshev's inequality we get the bound
\begin{eqnarray*}
  I(u,n) & \leq & \mathrm{P} \bigg[ \max_{1 \leq j \leq k_{n}} \sum_{k=1}^{j}
              \bigg| \frac{S_{r_{n}}^{k}}{a_{n}}
              1_{ \big\{ \frac{|S_{r_{n}}^{k}|}{a_{n}} \leq u  \big\}}
              - \mathrm{E} \bigg( \frac{S_{r_{n}}^{k}}{a_{n}}
              1_{ \big\{ \frac{|S_{r_{n}}^{k}|}{a_{n}} \leq u \big\} } \bigg) \bigg| > \delta\bigg]\\[0.4em]
  & = & \mathrm{P} \bigg[ \sum_{k=1}^{ k_{n}} \bigg|
             \frac{S_{r_{n}}^{k}}{a_{n}}
            1_{ \big\{ \frac{|S_{r_{n}}^{k}|}{a_{n}} \leq u  \big\} }
            - \mathrm{E} \bigg( \frac{S_{r_{n}}^{k}}{a_{n}} 1_{ \big\{ \frac{|S_{r_{n}}^{k}|}{a_{n}}
            \leq u \big\} } \bigg) \bigg| > \delta
            \bigg]\\[0.4em]
   & \leq & \delta^{-1} \mathrm{E} \bigg[ \sum_{k=1}^{k_{n}} \bigg|
            \frac{S_{r_{n}}^{k}}{a_{n}} 1_{ \big\{ \frac{|S_{r_{n}}^{k}|}{a_{n}} \leq u  \big\} }
            - \mathrm{E} \bigg( \frac{S_{r_{n}}^{k}}{a_{n}} 1_{ \big\{ \frac{|S_{r_{n}}^{k}|}{a_{n}}
            \leq u \big\} } \bigg) \bigg| \bigg]\\[0.4em]
   & \leq & 2 \delta^{-1} \sum_{k=1}^{k_{n}}
             \mathrm{E} \bigg( \frac{|S_{r_{n}}^{k}|}{a_{n}} 1_{ \big\{ \frac{|S_{r_{n}}^{k}|}{a_{n}}
            \leq u \big\} } \bigg)\\[0.4em]
   & = &  2 \delta^{-1} k_{n} \mathrm{E} \bigg( \frac{|S_{r_{n}}|}{a_{n}} 1_{ \big\{ \frac{|S_{r_{n}}|}{a_{n}}
            \leq u \big\} } \bigg).
 \end{eqnarray*}
  Using Lemma \ref{l:auxlemma1} we obtain that
  \begin{eqnarray*}
    \lim_{n \to \infty} k_{n} \mathrm{E} \bigg( \frac{|S_{r_{n}}|}{a_{n}}
   1 _{ \big\{ \frac{|S_{r_{n}}|}{a_{n}} \leq u \big\} } \bigg) &=&
    \int_{|x| \leq u} |x|\,\nu(dx) \\[0.5em]
     & = & (c_{-} + c_{+}) \frac{\alpha}{1-\alpha}
     u^{1-\alpha}\\[0.5em]
     & \to & 0 \qquad \textrm{as} \ u \to 0.
  \end{eqnarray*}
 Hence
\begin{equation*}
  \lim_{u \downarrow 0} \limsup_{n \to \infty} I(u, n)=0,
\end{equation*}
 which completes the proof, with the note that the $\alpha$--stability of the
 process $W_{0}(\,\cdot\,)$ follows from Theorem 14.3 in Sato \cite{Sa99} and the representation of the measure $\nu$ in
 (\ref{e:new-measure}).
\end{proof}

\smallskip

\begin{rem}
Theorem~\ref{t:newprocess} covers a wide range of stationary sequences. In Bartkiewicz et al.~\cite{BaJaMiWi09} are given some sufficient conditions for
relation (\ref{e:reg-var-new}) to hold (see their Theorem 1 and Section 3.2.2) and several examples of standard time series models that satisfy these conditions, including $m$--dependent sequences, GARCH(1,1) process and its squares, solutions to stochastic recurrence equations and stochastic volatility models. These conditions are:
\begin{itemize}
  \item[(C1)] The process $(X_{n})$ is regularly varying with index
  $\alpha \in (0,2)$, i.e. for every $d \geq 1$, the $d$--dimensional random vector $\mathbf{X} = (X_{1}, \ldots, X_{d})$ is multivariate regularly varying with index $\alpha$. This means that for some (and then for every) norm $\|\,\cdot\,\|$ on $\mathbb{R}^{d}$ there exists a random vector $\mathbf{\Theta}$  on the unit sphere $\mathbb{S}^{d-1} = \{ \mathbf{x} \in \mathbb{R}^{d} : \| \mathbf{x} \|=1 \}$ such that for every $u>0$ and as $x \to \infty$,
   \begin{equation*}
    \frac{\mathrm{P}(\|\mathbf{X}\| > ux, \mathbf{X} / \|\mathbf{X}\| \in \cdot\,)}{\mathrm{P}(\|\mathbf{X}\| >x)} \xrightarrow{w} u^{-\alpha} \mathrm{P}(\mathbf{\Theta} \in \cdot\,),
   \end{equation*}
   where the arrow "$\xrightarrow{w}$" denotes weak convergence of finite measures.
  \item[(C2)] There exists a sequence of positive integers $(r_{n})$ such that $r_{n} \to \infty$, $k_{n}= \lfloor n / r_{n} \rfloor \to \infty$ and for every $x \in \mathbb{R}$,
       \begin{equation*}
         \big| \varphi_{n}(x) - (\varphi_{n r_{n}}(x))^{k_{n}} \big|
        \to 0
       \end{equation*}
        as $n \to \infty$, where $\varphi_{nj}(x)= \mathrm{E}e^{ix a_{n}^{-1} S_{j}}$,
        $j=1,2, \ldots$, and $\varphi_{n}(x)=\varphi_{nn}(x)$.
  \item[(C3)] For every $x \in \mathbb{R}$,
       \begin{equation*}
         \lim_{d \to \infty} \limsup_{n \to \infty}
        \frac{n}{r_{n}} \sum_{j=d+1}^{r_{n}} \mathrm{E} \big| \overline{x
        a_{n}^{-1}(S_{j}-S_{d})} \cdot \overline{x a_{n}^{-1} X_{1}} \big|
        = 0,
       \end{equation*}
        where the sequence $(r_{n})_{n}$ is the same as in (C2) and for an arbitrary random variable $Z$ we put
        $\overline{Z} = (Z \wedge 2) \vee (-2)$.
  \item[(C4)] The limits
       \begin{equation*}
         \lim_{n \to \infty} n \mathrm{P}(S_{d}>a_{n})=b_{+}(d)
        \quad \textrm{and} \quad  \lim_{n \to \infty} n \mathrm{P}(S_{d} \leq
        -a_{n})=b_{-}(d), \qquad d \geq 1,
       \end{equation*}
       \begin{equation*}
         \lim_{d \to \infty}(b_{+}(d) - b_{+}(d-1))=c_{+} \quad
        \textrm{and} \quad \lim_{d \to \infty}(b_{-}(d) -
        b_{-}(d-1))=c_{-}
       \end{equation*}
        exists.
  \item[(C5)] For $\alpha >1$ assume $\mathrm{E} X_{1}=0$ and for $\alpha
        =1$,
       \begin{equation*}
         \lim_{d \to \infty} \limsup_{n \to \infty} n\,\big|
        \mathrm{E}(\sin (a_{n}^{-1}S_{d})) \big|=0.
       \end{equation*}
\end{itemize}

With appropriate (and standard) assumptions, which are precisely described in \cite{BaJaMiWi09}, the above mentioned time series models are strongly mixing with geometric rate, which suffices the mixing condition $\mathcal{A}^{*}(a_{n})$ to hold (see Proposition~\ref{p:strmix111} below). Therefore, for $\alpha \in (0,1)$, all conditions of Theorem~\ref{t:newprocess} are satisfied and the conclusion of the theorem follows. Naturally, for $\alpha \in [1,2)$ one has also to verify Condition \ref{c:step6cond-new}.
\end{rem}

\smallskip

\section{Appendix}\label{s:kraj}

In this section we give some sufficient conditions for
the mixing condition $\mathcal{A}^{*}(a_{n})$ and Condition
\ref{c:step6cond-new} to hold. These conditions are principally based on the well known strong or $\alpha$--mixing and $\rho$--mixing conditions. Let  $(\Omega,
\mathcal{F}, \mathrm{P})$ be a probability space. For any
$\sigma$-field $\mathcal{A} \subset \mathcal{F}$, let
$L_{2}(\mathcal{A})$ denote the space of square-integrable,
$\mathcal{A}$-measurable, real-valued random variables. For any two
$\sigma$-fields $\mathcal{A}, \mathcal{B} \subseteq \mathcal{F}$
define
\begin{equation*}
  \alpha (\mathcal{A}, \mathcal{B}) =
          \sup \{ | \mathrm{P}(A \cap B) - \mathrm{P}(A) \mathrm{P}(B) | : A \in \mathcal{A}, B \in \mathcal{B}
          \}
\end{equation*}
          and
\begin{equation*}
 \rho (\mathcal{A}, \mathcal{B}) =
        \sup \Big\{ \frac{|\mathrm{E}(XY) - \mathrm{E}X \mathrm{E}Y|}{\sqrt{\mathrm{E}X^{2} \mathrm{E}Y^{2}}} : X \in L_{2}(\mathcal{A}), Y
        \in L_{2}(\mathcal{B}) \Big\}.
\end{equation*}
Let now $(X_{n})_{n \in \mathbb{Z}}$ be a sequence of random
variables on $(\Omega, \mathcal{F}, \mathrm{P})$, and denote $\mathcal{F}_{k}^{l} = \sigma (\{ X_{i} : k \leq i
\leq l \})$ for $-\infty \leq k \leq l \leq \infty$. Then we say the sequence $(X_{n})_{n}$ is $\alpha$--mixing (or strongly mixing) if
\begin{equation*}
 \alpha(n) = \sup_{j \in \mathbb{Z}} \alpha(\mathcal{F}_{-\infty}^{j},
        \mathcal{F}_{j+n}^{\infty}) \to 0
\end{equation*}
        and $\rho$--mixing if
\begin{equation*}
 \rho(n) = \sup_{j \in \mathbb{Z}} \rho(\mathcal{F}_{-\infty}^{j},
        \mathcal{F}_{j+n}^{\infty}) \to 0
\end{equation*}
as $n \to \infty$. Note that when the sequence $(X_{n})$ is strictly stationary, one
has simply $\alpha(n) = \alpha (\mathcal{F}_{-\infty}^{0},
\mathcal{F}_{n}^{\infty})$, and similar for $\rho(n)$.

\begin{prop}\label{p:strmix111}
Suppose $(X_{n})$ is a strictly stationary sequence of regularly
varying random variables with index $\alpha \in (0,2)$, and
$(a_{n})$ a sequence of positive real numbers such that $n
\mathrm{P}(|X_{1}|>a_{n}) \to 1$ as $n \to \infty$. Assume relation
(\ref{e:reg-var-new}) holds for some sequence of positive integers
$(r_{n})$ such that $r_{n} \to \infty$ and $r_{n}/ n \to 0$ as $n \to \infty$, and $k_{n}= \lfloor n/r_{n}
\rfloor = o(n^{t})$ for some
$0<t<1$. If the sequence $(X_{n})$ is strongly mixing with
exponential rate, i.e. $\alpha_{n} \leq C \rho^{n}$ for some
$\rho \in (0,1)$ and $C>0$, where $(\alpha_{n})$ is the sequence of $\alpha$--mixing coefficients of $(X_{n})$, then the mixing condition
$\mathcal{A}^{*}(a_{n})$ holds.
\end{prop}

\begin{proof}
Let $(l_{n})$ be an arbitrary (and fixed) sequence of positive real numbers such that $l_{n} \sim n^{q}$, i.e. $l_{n} / n^{q} \to 1$ as $n \to \infty$, where $q=  \min \{ 1/\alpha, (1-t) / (1+\alpha) \} /2$.
Let $n$ be large enough such that $l_{n} < r_{n}$ (note that for
large $n$ it holds that $l_{n} < n^{1-t} < r_{n}$). We break $X_{1},
X_{2}, \ldots$ into blocks of $r_{n}$ consecutive random variables.
The last $l_{n}$ variables in each block will be dropped. Then we
shall show that doing so, the new blocks will be almost independent
(as $n \to \infty$) and this will imply relation
(\ref{e:mixcon-new}) for the new blocks. The error which occurs by
cutting of the ends of the original blocks will be small, and this
will imply condition (\ref{e:mixcon-new}) for the original blocks
also.

Take an arbitrary $f \in C_{K}^{+}(\mathbb{E})$. Since its support
is bounded away from $0$, there exists some $r>0$ such that $f(x) =
0$ for $ |x| \leq r$, and since $f$ is bounded, there exists
some $M>0$ such that $|f(x)| < M$ for all $x \in \mathbb{E}$. For
all $k,n \in \mathbb{N}$ define
\begin{equation*}
 S_{r_{n},\,l_{n}}^{k}= X_{kr_{n}-l_{n}+1}+ \ldots + X_{kr_{n}}.
\end{equation*}
$S_{r_{n},\,l_{n}}^{k}$ is the sum of the last $l_{n}$ random
variables in the $k$-th block. By stationarity we have
\begin{equation*}
 S_{r_{n}}^{k} - S_{r_{n},\,l_{n}}^{k} \eqd S_{r_{n}}^{1} - S_{r_{n},\,l_{n}}^{1}
   = S_{r_{n}-l_{n}}.
\end{equation*}
 This and the following inequality
\begin{equation*}
 |\mathrm{E} gh - \mathrm{E} g \mathrm{E} h| \leq 4 C_{1} C_{2} \alpha_{m},
\end{equation*}
for a $\mathcal{F}_{-\infty}^{j}$ measurable function $g$ and a
$\mathcal{F}_{j+m}^{\infty}$ measurable function $h$ such that $|g|
\leq C_{1}$ and $|h| \leq C_{2}$ (see Lemma 1.2.1 in Lin
and Lu \cite{LiLu97}), applied $k_{n}$ times, give
\begin{eqnarray}\label{e:mixineq}
  \nonumber \bigg| \mathrm{E} \exp \bigg( - \sum_{k=1}^{k_{n}}
   f(a_{n}^{-1}S_{r_{n}}^{k} - a_{n}^{-1}S_{r_{n},\,l_{n}}^{k}) \bigg) -
   \bigg( \mathrm{E} \exp (-f(a_{n}^{-1}S_{r_{n}-l_{n}}))
   \bigg)^{k_{n}} \bigg| & & \\[1.2em]
   & \hspace*{-56em} \leq & \hspace*{-28em} 4 k_{n} \alpha_{l_{n}+1}.
\end{eqnarray}
 Then
\begin{eqnarray}\label{e:ineqI}
  \nonumber \bigg| \mathrm{E} \exp \bigg( - \sum_{k=1}^{k_{n}}
              f(a_{n}^{-1}S_{r_{n}}^{k}) \bigg) - \bigg( \mathrm{E} \exp (-f(a_{n}^{-1}S_{r_{n}}))
              \bigg)^{k_{n}} \bigg| & & \\[0.3em]
  \nonumber  & \hspace*{-46em} \leq & \hspace*{-22em} \bigg| \mathrm{E} \exp \bigg( - \sum_{k=1}^{k_{n}}
   f(a_{n}^{-1}S_{r_{n}}^{k}) \bigg) -
   \mathrm{E} \exp \bigg (- \sum_{k=1}^{k_{n}} f(a_{n}^{-1}S_{r_{n}}^{k} - a_{n}^{-1}S_{r_{n},\,l_{n}}^{k}) \bigg)
   \bigg| \\[0.3em]
   \nonumber & & \hspace*{-22em} + \ \bigg| \mathrm{E} \exp \bigg( - \sum_{k=1}^{k_{n}}
   f(a_{n}^{-1}S_{r_{n}}^{k} - a_{n}^{-1}S_{r_{n},\,l_{n}}^{k}) \bigg) -
   \bigg( \mathrm{E} \exp (-f(a_{n}^{-1}S_{r_{n}-l_{n}}))
   \bigg)^{k_{n}} \bigg| \\[0.3em]
   \nonumber & & \hspace*{-22em} + \ \bigg| \bigg( \mathrm{E} \exp (-f(a_{n}^{-1}S_{r_{n}-l_{n}})) \bigg)^{k_{n}} -
   \bigg( \mathrm{E} \exp (-f(a_{n}^{-1}S_{r_{n}}))
   \bigg)^{k_{n}} \bigg|\\[0.9em]
   & \hspace*{-46em} =: & \hspace*{-22em} I_{1}(n) + I_{2}(n) + I_{3}(n).
\end{eqnarray}
 By Lemma 4.3 in Durrett \cite{Du96} and stationarity we have
 \begin{eqnarray*}
    \nonumber I_{1}(n) & \leq & \mathrm{E} \bigg( \sum_{k=1}^{k_{n}} |e^{-f(a_{n}^{-1}S_{r_{n}}^{k})} -
          e^{-f(a_{n}^{-1}S_{r_{n}}^{k} - a_{n}^{-1}S_{r_{n},\,l_{n}}^{k})}| \bigg)\\[0.5em]
  \nonumber & = & k_{n} \mathrm{E} \big|e^{-f(a_{n}^{-1}S_{r_{n}})} -
          e^{-f(a_{n}^{-1}S_{r_{n}-l_{n}})} \big|\\[0.5em]
  \nonumber & = & k_{n} \mathrm{E} \big| e^{-f(a_{n}^{-1}S_{r_{n}})}
         (1- e^{f(a_{n}^{-1}S_{r_{n}}) -f(a_{n}^{-1}S_{r_{n}-l_{n}})}) \big|  \\[0.5em]
  \nonumber & \leq & k_{n} \mathrm{E} \big|
         1- e^{f(a_{n}^{-1}S_{r_{n}}) -f(a_{n}^{-1}S_{r_{n}-l_{n}})} \big|.
\end{eqnarray*}
It can be shown that for any $t>0$ there exists a constant
$C=C(t)>0$ such that
\begin{equation*}
 |1-e^{-x}| \leq C |x|, \qquad \textrm{for all} \ |x| < t.
\end{equation*}
Since for all $x,y \in \mathbb{E}$, $ |f(x)-f(y)|<2M$, there exists
a positive constant $C$ such that
\begin{equation}\label{e:ineqI1}
    I_{1}(n) \leq Ck_{n} \mathrm{E} | f(a_{n}^{-1}S_{r_{n}}) - f(a_{n}^{-1}S_{r_{n}-l_{n}}) |.
\end{equation}
 Further,
\begin{eqnarray}\label{e:pomineqI1}
   \nonumber  \mathrm{E} | f(a_{n}^{-1}S_{r_{n}}) -
         f(a_{n}^{-1}S_{r_{n}-l_{n}})| & & \\[0.6em]
   \nonumber & \hspace*{-20em} = & \hspace*{-10em} \mathrm{E} \big[ |f(a_{n}^{-1}S_{r_{n}}) - f(a_{n}^{-1}S_{r_{n}-l_{n}})|
      1_{\{ a_{n}^{-1}|S_{r_{n}-l_{n}}| > r/2 \}}1_{ \{ a_{n}^{-1}|S_{r_{n}}|>r/4
      \}} \big]\\[0.6em]
   \nonumber & & \hspace*{-10em} + \ \mathrm{E} \big[ f(a_{n}^{-1}S_{r_{n}-l_{n}})
      1_{\{ a_{n}^{-1}|S_{r_{n}-l_{n}}| > r/2 \}}1_{ \{ a_{n}^{-1}|S_{r_{n}}| \leq r/4
      \}} \big]\\[0.6em]
   \nonumber & & \hspace*{-10em} + \ \mathrm{E} \big[ f(a_{n}^{-1}S_{r_{n}})1_{\{ a_{n}^{-1}|S_{r_{n}-l_{n}}| \leq r/2 \}}
        1_{ \{ a_{n}^{-1}|S_{r_{n}}|>r \}} \big]\\[0.6em]
   \nonumber & \hspace*{-20em} \leq & \hspace*{-10em} \mathrm{E} \big[
      |f(a_{n}^{-1}S_{r_{n}}) - f(a_{n}^{-1}S_{r_{n}-l_{n}})|
      1_{\{ a_{n}^{-1}|S_{r_{n}-l_{n}}| > r/2 \}}1_{ \{ a_{n}^{-1}|S_{r_{n}}|>r/4
      \}} \big]\\[0.6em]
     & & \hspace*{-10em} + \ M \mathrm{P} \bigg( \frac{|S_{l_{n}}|}{a_{n}} > \frac{r}{4} \bigg)
         + M \mathrm{P} \bigg( \frac{|S_{l_{n}}|}{a_{n}} > \frac{r}{2}
         \bigg).
\end{eqnarray}
Since the set $S= \{ x \in \mathbb{E} : |x| > r/4 \}$ is relatively
compact and any continuous function on a compact set is uniformly
continuous, it follows that for any $\epsilon >0$ there exists
$\delta >0$ such that $ |f(x) - f(y)| < \epsilon$ for all $x,y \in
S$ such that $\rho (x,y) \leq \delta$, where $\rho$ is the
metric on $\mathbb{E}$ defined in (\ref{e:metricrho}). If
$|x|>r/2$, $|y|>r/4$ and $\textrm{sign}(x)=\textrm{sign}(y)$, then
$x,y \in S$ and
\begin{equation}\label{e:metric1}
    \rho (x,y) = \frac{||x|-|y||}{|xy|} \leq \frac{8}{r^{2}}|x-y|.
\end{equation}
Define
\begin{equation*}
g_{n}(x,y)=|f(a_{n}^{-1}x)-f(a_{n}^{-1}y)|.
\end{equation*}
Let $\epsilon >0$ be arbitrary. Then
\begin{eqnarray*}
   \mathrm{E} \big[ |f(a_{n}^{-1}S_{r_{n}}) - f(a_{n}^{-1}S_{r_{n}-l_{n}})|
      1_{\{ a_{n}^{-1}|S_{r_{n}-l_{n}}| > r/2 \}}1_{ \{ a_{n}^{-1}|S_{r_{n}}|>r/4 \}} \big]
   & &  \\[0.5em]
   & \hspace*{-50em} = & \hspace*{-25em} \mathrm{E}
     \big[ g_{n}(S_{r_{n}}, S_{r_{n}-l_{n}})
      1_{\{ a_{n}^{-1}|S_{r_{n}-l_{n}}| > r/2,\,a_{n}^{-1}|S_{r_{n}}|>r/4 \}}
      1_{ \{ \textrm{sign}(S_{r_{n}-l_{n}})\,\neq\,\textrm{sign} (S_{r_{n}}) \} } \big]\\[0.5em]
   & \hspace*{-50em}  & \hspace*{-25em} + \ \mathrm{E} \big[ g_{n}(S_{r_{n}}, S_{r_{n}-l_{n}})
      1_{\{ a_{n}^{-1}S_{r_{n}-l_{n}} > r/2,\,a_{n}^{-1}S_{r_{n}}>r/4 \}}
      1_{ \{ a_{n}^{-1}|S_{r_{n}} - S_{r_{n}-l_{n}}| \leq \delta r^{2}/8 \} } \big]\\[0.5em]
   & \hspace*{-50em}  & \hspace*{-25em} +  \ \mathrm{E} \big[ g_{n}(S_{r_{n}}, S_{r_{n}-l_{n}})
      1_{\{ a_{n}^{-1}S_{r_{n}-l_{n}} < -r/2,\,a_{n}^{-1}S_{r_{n}} < -r/4 \}}
      1_{ \{ a_{n}^{-1}|S_{r_{n}} - S_{r_{n}-l_{n}}| \leq \delta r^{2}/8 \} } \big]\\[0.5em]
   & \hspace*{-50em}  & \hspace*{-25em} + \ \mathrm{E} \big[ g_{n}(S_{r_{n}}, S_{r_{n}-l_{n}})
      1_{\{ a_{n}^{-1}S_{r_{n}-l_{n}} > r/2,\,a_{n}^{-1}S_{r_{n}}>r/4 \}}
      1_{ \{ a_{n}^{-1}|S_{r_{n}} - S_{r_{n}-l_{n}}| > \delta r^{2}/8 \} } \big]\\[0.5em]
   & \hspace*{-50em}  & \hspace*{-25em} + \  \mathrm{E} \big[ g_{n}(S_{r_{n}}, S_{r_{n}-l_{n}})
      1_{\{ a_{n}^{-1}S_{r_{n}-l_{n}} < -r/2,\,a_{n}^{-1}S_{r_{n}} < -r/4 \}}
      1_{ \{ a_{n}^{-1}|S_{r_{n}} - S_{r_{n}-l_{n}}| > \delta r^{2}/8 \} } \big].
\end{eqnarray*}
By stationarity and relation (\ref{e:metric1}) this is bounded above
by
\begin{eqnarray*}
    &  \leq &  2M \mathrm{P} \bigg( \frac{|S_{r_{n}}-S_{r_{n}-l_{n}}|}{a_{n}}
      > \frac{3r}{4} \bigg)\\[0.6em]
   &  &  + \ \mathrm{E} \big[ g_{n}(S_{r_{n}}, S_{r_{n}-l_{n}})
      1_{\{ a_{n}^{-1}S_{r_{n}-l_{n}} > r/2 \}}1_{ \{ a_{n}^{-1}S_{r_{n}}>r/4 \}}
      1_{ \{ \rho(a_{n}^{-1}S_{r_{n}},\,a_{n}^{-1}S_{r_{n}-l_{n}}) \leq \delta  \} }
      \big]\\[0.6em]
    &  &  + \ \mathrm{E} \big[ g_{n}(S_{r_{n}}, S_{r_{n}-l_{n}})
      1_{\{ a_{n}^{-1}S_{r_{n}-l_{n}} < -r/2 \}}1_{ \{ a_{n}^{-1}S_{r_{n}}< -r/4 \}}
      1_{ \{ \rho(a_{n}^{-1}S_{r_{n}},\,a_{n}^{-1}S_{r_{n}-l_{n}}) \leq \delta  \} } \big]\\[0.6em]
   &  &  + \ 4M
      \mathrm{P} \bigg( \frac{|S_{r_{n}} - S_{r_{n}-l_{n}}|}{a_{n}} > \frac{\delta r^{2}}{8} \bigg)\\[0.6em]
   &  \leq &  2M \mathrm{P} \bigg(
       \frac{|S_{l_{n}}|}{a_{n}} > \frac{3r}{4} \bigg) + \epsilon
       \mathrm{P} \bigg( \frac{|S_{r_{n}}|}{a_{n}} > \frac{r}{4} \bigg) + 4M \mathrm{P} \bigg( \frac{|S_{l_{n}}|}{a_{n}} >
       \frac{\delta r^{2}}{8} \bigg).
\end{eqnarray*}
Therefore, from (\ref{e:ineqI1}) and (\ref{e:pomineqI1}) we obtain
\begin{equation}\label{e:ineqI1-0}
I_{1}(n) \leq 8 MCk_{n} \mathrm{P} \bigg(
\frac{|S_{l_{n}}|}{a_{n}} > \gamma \bigg) + \epsilon Ck_{n}
      \mathrm{P} \bigg( \frac{|S_{r_{n}}|}{a_{n}} > \frac{r}{4} \bigg),
\end{equation}
 where $\gamma = \min \{ r/4, \delta r^{2}/8 \}>0$.

Recall that, since $X_{1}$ is regularly varying with index $\alpha \in (0,2)$ it holds that
$\mathrm{P}(|X_{1}|>x) = x^{-\alpha} L(x)$ for any $x>0$,
where $L(\,\cdot\,)$ is a slowly varying function.  It also holds that
 $ a_{n} = n^{1/\alpha} L'(n)$,
 where $L'(\,\cdot\,)$ is a slowly varying function. Hence taking an arbitrary $0 < s < \min \{
\alpha,\,\alpha (1-t-q - \alpha q) /(1-\alpha q) \}$, we have
\begin{eqnarray*}
  k_{n} \mathrm{P} \bigg( \frac{|S_{l_{n}}|}{a_{n}} > \gamma \bigg)
   & \leq & k_{n} l_{n} \mathrm{P} (|X_{1}| > \gamma a_{n} / l_{n}) = k_{n}l_{n} \bigg(\frac{\gamma
   a_{n}}{l_{n}} \bigg)^{-\alpha} L \bigg( \frac{\gamma
   a_{n}}{l_{n}} \bigg) \\
   &=& k_{n}l_{n} \bigg(\frac{\gamma
   a_{n}}{l_{n}}\bigg)^{s - \alpha} \cdot c_{n},
\end{eqnarray*}
 where
\begin{equation*}
 c_{n} = \bigg( \frac{\gamma a_{n}}{l_{n}} \bigg)^{-s} L \bigg(
    \frac{\gamma a_{n}}{l_{n}} \bigg).
\end{equation*}
  Since $a_{n} / l_{n} \to \infty$
as $n \to \infty$, by Proposition 1.3.6 in Bingham et al.
\cite{BiGoTe89} we have that
 $c_{n} \to 0$ as $n \to \infty$. Further
\begin{eqnarray*}
   k_{n}l_{n} \bigg(\frac{\gamma a_{n}}{l_{n}}\bigg)^{s - \alpha} & = &
       \frac{k_{n}(l_{n})^{1+\alpha -s}}{\gamma^{\alpha -s}a_{n}^{\alpha
       -s}} = \bigg(\frac{l_{n}}{n^{q}}\bigg)^{1+\alpha -s} \cdot \frac{k_{n}}{n^{t}} \cdot
       \frac{n^{t}(n^{q})^{1+\alpha -s}}{\gamma^{\alpha -s}n^{(\alpha
       -s)/\alpha}(L'(n))^{\alpha-s}}\\[0.4em]
   & \leq & \bigg(\frac{l_{n}}{n^{q}}\bigg)^{1+\alpha -s} \cdot
       \frac{k_{n}}{n^{t}} \cdot \frac{1}{\gamma^{\alpha -s}n^{p}(L'(n))^{\alpha-s}}
\end{eqnarray*}
where
 $p = (\alpha -s) /\alpha - t - (1+\alpha-s)q$. It can easily be
checked that $p>0$. This and the fact that $l_{n} \sim n^{q}$ and
$k_{n}=o(n^{t})$, by Proposition 1.3.6 in Bingham et al.
\cite{BiGoTe89}, imply that $k_{n}l_{n} (\gamma a_{n} / l_{n})^{s -
\alpha}
   \to 0$ as $n \to \infty$. Hence
\begin{equation}\label{e:ineqI1-1}
    k_{n} \mathrm{P} \bigg( \frac{|S_{l_{n}}|}{a_{n}} > \gamma \bigg) \to
    0
    \quad \textrm{as} \ n \to \infty.
\end{equation}
From relation (\ref{e:reg-var-new}) we obtain that, as $n \to
\infty$,
 \begin{equation}\label{e:ineqI1-2}
    k_{n} \mathrm{P} \bigg( \frac{|S_{r_{n}}|}{a_{n}} > \frac{r}{4} \bigg)
   \to  \nu ( \{x \in \mathbb{E} : |x| > r/4 \}) =:A < \infty.
 \end{equation}
 Thus from relations (\ref{e:ineqI1-0}),
(\ref{e:ineqI1-1}) and (\ref{e:ineqI1-2}) we obtain
\begin{equation*}
 \limsup_{n \to \infty} I_{1}(n) \leq AC \epsilon,
\end{equation*}
and since $\epsilon >0$ is arbitrary, we have
\begin{equation}\label{e:1ineqI1}
    \lim_{n \to \infty} I_{1}(n)=0.
\end{equation}
From the assumption that $(X_{n})$ is strongly mixing with exponential rate it follows that $k_{n}
\alpha_{l_{n}+1} \to 0$ as $n \to \infty$, and hence from (\ref{e:mixineq}) we obtain
\begin{equation}\label{e:1ineqI2}
    \lim_{n \to \infty} I_{2}(n)=0.
\end{equation}
Using again Lemma 4.3 in Durrett \cite{Du96} it follows
\begin{equation*}
  I_{3}(n) \leq  k_{n} \mathrm{E} \big| e^{-f(a_{n}^{-1}S_{r_{n}})} -
          e^{-f(a_{n}^{-1}S_{r_{n}-l_{n}})} \big|.
\end{equation*}
Repeating the same procedure as for $I_{1}(n)$ we get
\begin{equation}\label{e:1ineqI3}
    \lim_{n \to \infty} I_{3}(n)=0.
\end{equation}
Taking into account relations (\ref{e:1ineqI1}), (\ref{e:1ineqI2})
and (\ref{e:1ineqI3}), from (\ref{e:ineqI}) we obtain that, as $n
\to \infty$,
\begin{equation*}
 \mathrm{E} \exp \bigg( - \sum_{k=1}^{k_{n}}
              f(a_{n}^{-1}S_{r_{n}}^{k}) \bigg) - \bigg( \mathrm{E} \exp (-f(a_{n}^{-1}S_{r_{n}}))
              \bigg)^{k_{n}} \to 0,
\end{equation*}
 and this concludes the proof.
\end{proof}

%\smallskip

\begin{prop}\label{p:rhomix}
Suppose $(X_{n})$ is a strictly stationary sequence of
regularly varying random variables with index of regular variation
$\alpha \in (1,2)$, and $(a_{n})$ a sequence of positive real
numbers such that $n \mathrm{P}(|X_{1}|>a_{n}) \to 1$ as $n \to
\infty$. Let $(r_{n})$ be sequence of positive
integers such that $r_{n} \to \infty $ as $n \to \infty$. If $r_{n} = o(n^{s})$ for some $0< s < 2/\alpha -1$, and the sequence $(\rho_{n})$ of $\rho$--mixing coefficients
of
 $(X_{n})$ decreases to zero as $n \to \infty$ and
 \begin{equation}\label{e:new-series-rho}
    \sum_{j \geq 0} \rho_{\lfloor 2^{j/3} \rfloor} < \infty,
 \end{equation}
 then Condition \ref{c:step6cond-new} holds.
\end{prop}

\begin{proof}
 Let $n \in \mathbb{N}$ and $u >0$ be arbitrary.
Define
\begin{equation*}
 Z_{k} = Z_{k}(u, n) = \frac{S_{r_{n}}^{k}}{a_{n}} 1_{ \big\{ \frac{|S_{r_{n}}^{k}|}{a_{n}}
   \leq u  \big\}} - \mathrm{E} \Big( \frac{S_{r_{n}}^{k}}{a_{n}}
    1_{ \big\{ \frac{|S_{r_{n}}^{k}|}{a_{n}} \leq u \big\} }
    \Big), \qquad k \in \mathbb{N}.
 \end{equation*}
Take an arbitrary $\delta >0$ and as in the proof of Theorem
\ref{t:newprocess} define
\begin{equation*}
 I(u,n) =
    \mathrm{P} \bigg[ \max_{1 \leq j \leq k_{n}} \bigg| \sum_{k=1}^{j} Z_{k} \Big| > \delta \Big].
\end{equation*}
Corollary 2.1 in Peligrad \cite{Pe99} then implies
\begin{equation*}
 I(u,n) \leq \delta^{-2} C  \exp \Big( 8 \sum_{j=0}^{\lfloor \log_{2} k_{n}
\rfloor} \widetilde{\rho}_{\lfloor 2^{j/3} \rfloor} \Big)\,k_{n}
\mathrm{E} (Z_{1}^{2}),
\end{equation*}
where $(\widetilde{\rho}_{k})$ is the sequence of
$\rho$-mixing coefficients of $(Z_{k})$ and $C$ is some positive
constant (here we put $\log_{2}0:=0$).
 Now standard calculations show that for any $k \in
\mathbb{N}$,
\begin{equation*}
 \widetilde{\rho}_{k} \leq
\rho_{(k-1)r_{n}+1},
\end{equation*}
and since the sequence $(\rho_{k})$ is
non-increasing, we have
 $ \widetilde{\rho}_{k} \leq \rho_{k}$. From this and
 assumption (\ref{e:new-series-rho}) we obtain that
 \begin{equation}\label{e:rho-new1}
    I(u,n) \leq CL \delta^{-2}\,k_{n}
     \mathrm{E} (Z_{1}^{2}),
 \end{equation}
 for some positive constant $L$.
Further we have
\begin{eqnarray}\label{e:rhomixcond2}
  \nonumber \mathrm{E} (Z_{1}^{2}) & \leq &   \mathrm{E} \Big( \frac{|S_{r_{n}}|^{2}}{a_{n}^{2}}
      1_{ \big\{ \frac{|S_{r_{n}}|}{a_{n}} \leq u  \big\}} \Big) =
      \mathrm{E} \Big( \frac{|S_{r_{n}}|^{2}}{a_{n}^{2}}
      1_{ \big\{ \frac{|S_{r_{n}}|}{a_{n}} \leq u  \big\}}
      1_{ \{ \cap_{i=1}^{r_{n}} \{ |X_{i}| \leq u a_{n} \} \} }  \Big) \\[0.7em]
  \nonumber &  & + \ \mathrm{E} \Big( \frac{|S_{r_{n}}|^{2}}{a_{n}^{2}}
      1_{ \big\{ \frac{|S_{r_{n}}|}{a_{n}} \leq u  \big\}}
      1_{ \{ \cup_{i=1}^{r_{n}} \{ |X_{i}| > u a_{n} \} \} }  \Big) \\[0.7em]
     & \leq &  \mathrm{E} \Big( \Big| \sum_{i=1}^{r_{n}}
       \frac{X_{i}}{a_{n}} 1_{\big\{ \frac{|X_{i}|}{a_{n}} \leq u \big\} } \Big|^{2} \Big) +
        u^{2} \mathrm{P} \Big( \bigcup_{i=1}^{r_{n}} \{ |X_{i}| > u a_{n} \}
        \Big).
\end{eqnarray}
Note that
\begin{eqnarray}\label{e:peligradrastav}
  \nonumber \mathrm{E} \Big( \Big| \sum_{i=1}^{r_{n}} \frac{X_{i}}{a_{n}} 1_{ \big\{ \frac{|X_{i}|}{a_{n}} \leq u \big\} }
    \Big|^{2} \Big) & &\\[0.6em]
     \nonumber & \hspace*{-14em} = & \hspace*{-7em} \mathrm{E} \Big( \Big| \sum_{i=1}^{r_{n}} \frac{X_{i} 1_{ \{ |X_{i}| \leq u a_{n} \}} - \mathrm{E}(X_{i} 1_{ \{ |X_{i}| \leq u a_{n} \}} )}{a_{n}}
      + \sum_{i=1}^{r_{n}} \mathrm{E} \Big( \frac{X_{i}}{a_{n}} 1_{ \big\{ \frac{|X_{i}|}{a_{n}} \leq u \big\} } \Big) \Big|^{2} \Big)\\[0.5em]
   &  \hspace*{-14em} = & \hspace*{-7em} \mathrm{E}(I_{1}^{2}) + 2 \mathrm{E}(I_{1}) I_{2} + I_{2}^{2},
\end{eqnarray}
where
\begin{equation*}
 I_{1}= \sum_{i=1}^{r_{n}} \frac{X_{i} 1_{ \{ |X_{i}| \leq u a_{n} \}} - \mathrm{E}(X_{i} 1_{ \{ |X_{i}| \leq u a_{n} \}} )}{a_{n}} \quad \textrm{and} \quad I_{2}= \sum_{i=1}^{r_{n}} \mathrm{E} \Big( \frac{X_{i}}{a_{n}} 1_{ \big\{ \frac{|X_{i}|}{a_{n}} \leq u \big\} } \Big).
\end{equation*}
Since $I_{1}$ is a sum of centered random variables, by Theorem 2.1 in
Peligrad \cite{Pe99} we have
\begin{equation}\label{e:tmpeligrad1}
    \mathrm{E} (I_{1}^{2}) \leq C \exp \Big( 8 \sum_{j=0}^{\lfloor \log_{2} r_{n} \rfloor}
    \rho_{ \lfloor 2^{j/3} \rfloor} (n,u) \Big)\,r_{n} \mathrm{E} \Big( \frac{X_{1}^{2}}{a_{n}^{2}}
    1_{ \big\{ \frac{|X_{1}|}{a_{n}} \leq u \big\} } \Big),
\end{equation}
for all $n \in \mathbb{N}$, where $(\rho_{j}(n, u))_{j}$ is the
sequence of $\rho$-mixing coefficients of
 $\displaystyle \Big( \frac{X_{j}}{a_{n}} 1_{\big\{ \frac{|X_{j}|}{a_{n}} \leq
  u \big\} } - \mathrm{E} \Big( \frac{X_{j}}{a_{n}} 1_{\big\{ \frac{|X_{j}|}{a_{n}} \leq
  u \big\} } \Big) \Big)_{j}$.
  Since the function
 $f = f_{n,u} \colon \mathbb{R} \to \mathbb{R}$ defined by
\begin{equation*}
 f(x) = \frac{x}{a_{n}} 1_{ \big\{ \frac{|x|}{a_{n}} \leq u \big\} } - \mathrm{E} \Big( \frac{X_{1}}{a_{n}} 1_{\big\{ \frac{|X_{1}|}{a_{n}} \leq
  u \big\} } \Big)
\end{equation*}
 is measurable, it follows that
\begin{equation*}
 \sigma \Big( \frac{X_{j}}{a_{n}}
 1_{ \big\{ \frac{|X_{j}|}{a_{n}} \leq u \big\} } - \mathrm{E} \Big( \frac{X_{j}}{a_{n}} 1_{\big\{ \frac{|X_{j}|}{a_{n}} \leq
  u \big\} } \Big) \Big) \subseteq \sigma (X_{j})
\end{equation*}
 (see Theorem 4 in Chow and Teicher \cite{ChTe97}). From this we immediately obtain
 $ \rho_{j}(n,u) \leq \rho_{j},$
for all $j, n \in \mathbb{N}$ and $u>0$.  Thus from
(\ref{e:tmpeligrad1}), by a new application of assumption
(\ref{e:new-series-rho}), we get
\begin{equation}\label{e:tmpeligrad2-new}
  \mathrm{E} (I_{1}^{2}) \leq CL\,r_{n} \mathrm{E} \Big( \frac{X_{1}^{2}}{a_{n}^{2}}
    1_{ \big\{ \frac{|X_{1}|}{a_{n}} \leq u \big\} } \Big).
\end{equation}
Note also $\mathrm{E}(I_{1})=0$. Since $\alpha \in (1,2)$ it holds that $\mathrm{E}|X_{i}| < \infty$. Hence
\begin{equation}\label{e:peligradnovo}
    I_{2}^{2} \leq M \frac{r_{n}^{2}}{a_{n}^{2}},
\end{equation}
for some positive constant $M$.
Now relations (\ref{e:rhomixcond2}), (\ref{e:peligradrastav}), (\ref{e:tmpeligrad2-new}) and (\ref{e:peligradnovo})
imply
\begin{eqnarray*}
   \nonumber k_{n} \mathrm{E} (Z_{1}^{2}) & \leq &  C L\,k_{n} r_{n} \mathrm{E} \Big( \frac{X_{1}^{2}}{a_{n}^{2}}
    1_{ \big\{ \frac{|X_{1}|}{a_{n}} \leq u \big\} } \Big) + M \frac{k_{n} r_{n}^{2}}{a_{n}^{2}} +  u^{2} k_{n}r_{n} P(|X_{1}| > u
    a_{n})\\[0.5em]
    & = & u^{2} \cdot \frac{k_{n}r_{n}}{n} \cdot n
    \mathrm{P}(|X_{1}|>ua_{n}) \cdot \bigg[ C L \frac{ \mathrm{E}[X_{1}^{2}
      1_{\{ |X_{1}| \leq u a_{n} \}}]}{u^{2} a_{n}^{2}P(|X_{1}|>u a_{n})} + 1
      \bigg] + M \frac{k_{n} r_{n}^{2}}{a_{n}^{2}}.
\end{eqnarray*}
From this, using the regular variation property of $X_{1}$, Karamata's theorem and the fact that $k_{n}r_{n}/n \to 1$ and
\begin{equation*}
 \frac{k_{n} r_{n}^{2}}{a_{n}^{2}} = \frac{k_{n}r_{n}}{n} \cdot \frac{r_{n}}{n^{s}} \cdot \frac{1}{n^{2/\alpha -1 -s}(L'(n))^{2}} \to 0
\end{equation*}
 as $n \to
\infty$ (here we used again the representation $a_{n}= n^{1/\alpha}L'(n)$, with $L'(\,\cdot\,)$ being a slowly varying function at $\infty$)
 we obtain
 \begin{equation*}
 \limsup_{n \to \infty}\,k_{n} \mathrm{E} (Z_{1}^{2}) \leq
u^{2-\alpha} \Big( \frac{CL \alpha}{2-\alpha} +1 \Big).
 \end{equation*}
 Letting $u \downarrow 0$, it follows that $\lim_{u \downarrow 0}
 \limsup_{n \to \infty}\,k_{n} \mathrm{E} (Z_{1}^{2}) =0$.
 Therefore, from (\ref{e:rho-new1}), we get
\begin{equation*}
  \lim_{u \downarrow 0} \limsup_{n \to \infty} I(u,n)=0,
\end{equation*}
 and Condition \ref{c:step6cond-new} holds.
\end{proof}

\begin{rem}
A careful analysis of the proof of Proposition~\ref{p:rhomix} shows that the additional condition on the sequence $(r_{n})$ (namely, $r_{n} = o(n^{s})$ for some $0< s < 2/\alpha -1$) can be dropped if we assume that the random variables $X_{i}$ are symmetric, since then we can directly apply Theorem 2.1 of Peligrad~\cite{Pe99} to $ \mathrm{E} ( | \sum_{i=1}^{r_{n}} a_{n}^{-1}X_{i} 1_{ \{ |X_{i}| \leq u a_{n}\} }|^{2} )$, and hence we do not need to introduce $I_{2}$ (to which the additional condition on $r_{n}$ is related). This also holds for $\alpha =1$.
\end{rem}

%\bibliographystyle{amsplain}
%%
% requires a BiBTeX file sample.bib
%\bibliography{sample}

\end{document}